\documentclass[12pt]{article}
\usepackage{amsfonts,amssymb,color,url}
\usepackage{blkarray}
\usepackage{longtable}
\usepackage[all,knot,arc]{xy}

\newtheorem{theorem}{Theorem}[section]
\newtheorem{prop}[theorem]{Proposition}
\newtheorem{problem}[theorem]{Problem}
\newtheorem{lemma}[theorem]{Lemma}

\newtheorem{unnumber}{}

\newtheorem{prepf}[unnumber]{Proof}
\newenvironment{proof}{\prepf\rm}{\endprepf}
\newcommand{\qed}{\hfill$\Box$}
\newtheorem{preconj}[unnumber]{Conjecture}
\newenvironment{conj}{\preconj\rm}{\endpreconj}
\newtheorem{preqn}[unnumber]{Question}
\newenvironment{qn}{\preqn\rm}{\endpreqn}
\newtheorem{predef}[unnumber]{Definition}
\newenvironment{defn}{\predef\rm}{\endpredef}

\newcommand{\pom}{\mathop{\mathrm{P}\Omega}}
\newcommand{\agl}{\mathop{\mathrm{AGL}}}
\newcommand{\asl}{\mathop{\mathrm{ASL}}}
\newcommand{\psl}{\mathop{\mathrm{PSL}}}
\newcommand{\psigmal}{\mathop{\mathrm{P\Sigma L}}}
\newcommand{\pgl}{\mathop{\mathrm{PGL}}}
\newcommand{\pgaml}{\mathop{\mathrm{P}\Gamma\mathrm{L}}}
\newcommand{\agaml}{\mathop{\mathrm{A}\Gamma\mathrm{L}}}
\newcommand{\psigl}{\mathop{\mathrm{P}\Sigma\mathrm{L}}}
\newcommand{\M}{\mathord{\mathrm{M}}}
\newcommand{\Co}{\mathord{\mathrm{Co}}}
\newcommand{\Sz}{\mathop{\mathrm{Sz}}}
\newcommand{\HS}{\mathord{\mathrm{HS}}}
\newcommand{\Sp}{\mathop{\mathrm{Sp}}}
\newcommand{\GF}{\mathop{\mathrm{GF}}}

\newcommand{\sym}{S_n}
\newcommand{\alt}{A_n}

\begin{document}

\title{Primitive groups, road closures, and idempotent generation}
\author{Jo\~ao Ara\'ujo and Peter J. Cameron}
\date{}
\maketitle

\begin{abstract}
We are interested in semigroups of the form $\langle G,a\rangle\setminus G$,
where $G$ is a permutation group of degree $n$ and $a$ a non-permutation on
the domain of $G$. A theorem of the first author, Mitchell and Schneider
shows that, if this semigroup is idempotent-generated for all possible choices
of $a$, then $G$ is the symmetric or alternating group of degree $n$, with
three exceptions (having $n=5$ or $n=6$). Our purpose here is to prove stronger
results where we assume that $\langle G,a\rangle\setminus G$ is 
idempotent-generated for all maps of fixed rank $k$. For $k\ge6$ and
$n\ge2k+1$, we reach the same conclusion, that $G$ is symmetric or alternating.
These results are proved using a stronger version of the
\emph{$k$-universal transversal property} previously considered by the authors.

In the case $k=2$, we show that idempotent generation of the semigroup for all
choices of $a$ is equivalent to a condition on the permutation group $G$,
stronger than primitivity, which we call the \emph{road closure condition}.
We cannot determine all the primitive groups with this property, but we give
a conjecture about their classification, and a body of evidence (both
theoretical and computational) in support of the conjecture. 

The paper ends with some problems.
\end{abstract}

\section{Introduction}

Let $S_n$ be the symmetric group and $T_n$ be the full transformation monoid on $n$ points. Let $T_{n,k}$, for $k\le n$, denote the set of rank $k$ transformations in $T_n$. Given a semigroup $S$ we denote by $E(S)$ its set of idempotents.

Our goal is
to classify the permutation groups that together with any rank $k$ map (for $1\le k\le n/2$) generate an idempotent generated semigroup of singular maps. Thus we aim at  classifying the groups $G\le S_n$ such that
\[
(\forall t\in T_{n,k})\ S:=\langle G,t\rangle\setminus G=\langle E(S)\rangle.
\]

We say that a group satisfying this condition  has the {\em $k$-id property}. This property is very difficult to work with and hence we introduce two auxiliary conditions, one necessary and the other sufficient. The first  is called the \emph{$k$-ut  property} (short for \emph{$k$-universal transversal property}), and is defined as follows: a primitive group possesses it if in the orbit of any $k$-set contained in $\Omega$ there is a transversal for every $k$-partition of $\Omega$. The second  is called the \emph{strong $k$-ut  property} and a primitive group possesses it if given any $(k+1)$-tuple $(a_1,\ldots,a_{k+1})$ of pairwise different elements of $\Omega$, and given any $k$-partition $P$ of $\Omega$, there exists $g\in G$ such that $\{a_1,\ldots,a_k\}g$ is a transversal for $P$, and $a_1g, a_{k+1}g$ belong to the same part of $P$.  We prove that the folowing implications hold (for $k\le n/2$):
\[
\mbox{$k$-homogeneity and strong $k$-ut}\Rightarrow \mbox{$k$-id}\Rightarrow \mbox{$k$-ut }.
\]

Modulo a few exceptional groups and families of groups, all groups with the $k$-ut property are $k$-homogeneous. Therefore, we expected the gap between the smallest and largest (more tractable) classes above to be very small; and the property of interest, though less tractable, would be within this gap. In addition, we have  the classification of groups with the $k$-ut property \cite{ac_tams}. Giving the general picture, the prospects of success in classifying the groups possessing the $k$-id property seemed high, but reality turned out to be much more interesting! Even if the two extreme properties look very close to each other, the fact is that this is not enough to decide all groups, with the case  $k=3$ standing out as particularly difficult. 

The approach outlined above worked pretty  well for the case of $k\ge 4$ yielding the following results:

\begin{theorem}
Let $G$ be a permutation group of degree $n$.
\begin{enumerate}
\item Suppose that $k\ge6$ and $n\ge2k+1$. Then $G$ has the $k$-id property
if and only if $G$ is $S_n$ or $A_n$.
\item Suppose that $n\ge11$. Then $G$ has the $5$-id property if and only if
$G$ is $S_n$, $A_n$, $M_{12}$ (with $n=12$) or $M_{24}$ (with $n=24$),
or possibly $\pgaml(2,32)$ (with $n=33$).
\item Suppose that $n\ge11$. Then $G$ has the $4$-id property if and only if
$G$ is $S_n$, $A_n$, $M_n$ (with $n=11,12,23,24$), or possibly 
$\psl(2,q)\le G\le\pgaml(2,q)$ with either $q$ prime congruent to
$11\pmod{12}$ or $q=2^p$ with $p$ prime, or $G=M_{11}$ with $n=12$.
\end{enumerate}
\end{theorem}

We also classified the groups with the $k$-id for degrees smaller than $11$, but the list is too long to be included here.
 
\medskip

The case $k=2$ is not amenable to this approach (in part since the $2$-ut
property is equivalent to primitivity). So we took a different approach, as
outlined in the next two results.

\begin{theorem}\label{x}
Let $G$ be a primitive group acting on $\Omega$, and $t$ a rank $2$ map. Then $\langle G, t\rangle\setminus G$ is idempotent generated if and only the bipartite graph whose vertices are elements in the orbits of $\ker(t)$ and $\Omega t$, with a set $S$ and a partition $P$ forming an edge whenever $S$ is a transversal for $P$, is connected.
\end{theorem}
 Observe that this result does not deal with the stronger property of $2$-id; rather, it is about any primitive group and any rank $2$ map. A classification of the pairs $(G,t)$ so that $\langle G,t\rangle \setminus G$ is idempotent generated is probably beyond reach now, but at least the theorem above says where to look for. 
 
\begin{theorem}\label{roadc}\label{t:main}
Let $G$ be a finite transitive permutation group on $\Omega$. The following
two conditions are equivalent:
\begin{enumerate}\itemsep0pt
\item $G$ has the $2$-id property;
\item for every orbit $O$ of $G$ on $2$-sets of $\Omega$, and every maximal
block of imprimitivity $B$ for $G$ acting on $O$, the graph with vertex set
$\Omega$ and edge set $O\setminus B$ is connected.
\end{enumerate}
\end{theorem}
 
The condition (b) in the above theorem is called the \emph{road closure property}, for reasons to be explained later.

Recall that a permutation group $G$ is:
\begin{itemize}
\item \emph{transitive}, if it preserves no non-empty proper subset of
$\Omega$;
\item \emph{primitive}, if it preserves no non-trivial partition of $\Omega$;
\item \emph{basic}, if it is primitive and also preserves no non-trivial
Cartesian power structure on $\Omega$.
\end{itemize}
The \emph{O'Nan--Scott Theorem} asserts, in part, that a basic primitive
group is affine, diagonal, or almost simple.
 
 \begin{theorem}\label{nonex1}
\begin{enumerate}\itemsep0pt
\item A transitive imprimitive group fails the road closure property.
\item A primitive non-basic group fails the road closure property.
\item A primitive group which has an imprimitive normal subgroup of index~$2$
fails the road closure property.
\item The primitive action of $\pom^+(8,q):S_3$ (described on page \pageref{trial})
fails the road closure property.
\end{enumerate}
\end{theorem}

\begin{conj}
A primitive basic permutation group which does not satisfy condition (c) or (d)
of Theorem~\ref{nonex1} has the road closure property.
\end{conj}
We tried very hard to prove this conjecture (and posed it to some top experts in permutation groups too), but without success.  Nevertheless this conjecture seems very interesting, and is the last obstacle before the complete classification of groups with the $2$-id property. 

\medskip

The results above are a dramatic generalization of the following theorem whose context we explain below. 

\begin{theorem}\label{main}
If $n\geq 1$ and $G$ is a subgroup of $\sym$, then the following are equivalent:
\begin{enumerate}\itemsep0pt
\item[(i)]
The semigroup $\langle G,a\rangle \setminus G$ is idempotent generated for all $a\in T_n \setminus\sym$.
\item[(ii)] One of the following is valid for $G$ and $n$:
\begin{enumerate}\itemsep0pt
\item[(a)] $n=5$ and $G\cong \agl(1,5)$;
\item[(b)] $n=6$ and $G\cong \psl(2,5)$ or $\pgl(2,5)$;
\item[(c)] $G=A_n$ or $\sym$.
\end{enumerate}
\end{enumerate}
\end{theorem}

Two of the most famous results in semigroup theory are due to Howie \cite{howie} and Erdos \cite{erdos} and deal with idempotent generated semigroups. Howie proved that the semigroups 
 $\langle S_n,t\rangle\setminus S_n$, where $t$ is a  transformation of rank $n-1$, are idempotent generated; 
 Erdos proved an analogue for linear transformations of a finite dimension vector space. Together the papers \cite{erdos} and \cite{howie} are cited in 
over    one hundred articles,  dealing with subjects including 
semigroups, groups, universal algebra,    ring theory, topology, and  combinatorics.   
Since the publication,    various different proofs for these results  
have  appeared. And in fact, one of the fundamental trends in
 semigroup theory has been the study of how idempotents shape the structure of the semigroup. Howie's book \cite{Ho95} can be seen as an excellent survey of the results obtained from the 40s to the 90s  on this general problem. 
 
Evidently, there is the analogous question for the group of units, namely, to what extent the group of units shapes the structure of the semigroup.  By the time the results above were proved, it was already clear that one of the fundamental aspects of 
the study of finite semigroups is the interplay between groups and idempotents. Many examples among the most famous structural theorems in semigroup 
theory, such as, the Rees Theorem~\cite{rees1940}, or McAlister's $P$-Theorem~\cite[Part II, Theorem~2.6]{mcal1974} (see also~\cite{munn1976}),  show the large extent to which the structure of a semigroup is shaped by a 
group acting in some way on an idempotent structure.

However, unlike the idempotents case, the group of units approach quickly leads to problems that could not be tackled with the tools available 30 years ago, let alone 70 years ago. Fortunately now the situation is totally different, since the enormous  progress made  in the last decades in the theory of permutation groups  provides   the necessary tools to develop semigroup theory from this different point of view.

It was in this general environment  that Theorem \ref{main} has been proved, a dramatic generalization of the results available by then, and also one of the first results on transformation semigroups that required the classification of finite simple groups. 

After the proof of Theorem \ref{main} the great challenge was the classification of the $k$-id groups, and that has been occupying the two authors for the last $8$ years, a project that now finishes. Of course it was impossible to tackle this project before classifying the groups possessing the $k$-ut property and, as that, reference \cite{ac_tams} can be seen as a mere lemma for the results in this paper. We would like to express our gratitude to P. M. Neumann (Oxford) for the support he gave to this long project, since its very beginning. We also thank the comments R. Gray (East Anglia) made on Section \ref{corners}; his deep insight into this topic was very important to us. 

\medskip

We now outline the content of the paper.

Section \ref{prelim} contains a number of results from the literature that will be needed in the reminder of the paper. In Section \ref{sktu} we prove that a $k$-homogeneous group with the strong $k$-ut property possesses the $k$-id property. 
 In Section \ref{sutph} we investigate the connections between the strong $k$-ut property and homogeneity. (Of course the ultimate goal would be to prove that the two properties coincide, or one contains the other, but we could reach no conclusion regarding that.) Additionally,  we provide the classification of groups that have the $k$-id property for  $k\ge 4$. For groups of degree at least $11$, the classification is almost complete, with a small set of well identified groups remaining to be decided. Section \ref{11s} provides the classification of all the groups of degree at most $10$ that possess the $k$-id property. Together with the results of the previous section this finishes the case of $k\ge4$.

Section \ref{corners} contains the proof of Theorem \ref{x}. Section \ref{roadclos} contains the proof  of Theorem \ref{t:main}, and also the main conjecture of this paper. Section \ref{pos} contains examples of groups that have the $2$-id property. Section \ref{comput} contains some results on the computations. The paper finishes with some open problems.
 
\section{The universal transversal property}\label{prelim}

As a preliminary to our main result, we show that the existence of rank $k$
idempotents in $\langle a,G\rangle\setminus G$, where $a$ is any map of
rank~$k$, is equivalent to a property of $G$ studied in \cite{ac_tams}.

\begin{defn}
Let $k$ be a positive integer less than $n$. Recall that the permutation group $G$ of
degree $n$ has the \emph{$k$-universal transversal property} (or
$k$-ut property, for short) if, given any $k$-subset $S$ of the domain of $G$ and any
$k$-part partition $P$ of the domain, there is an element $g\in G$ such
that $Sg$ is a section (or transversal) for $P$.
\end{defn}

\begin{theorem}\label{utid}
For a permutation group $G$ of degree $n$, and an integer $k<n$, the following
are equivalent:
\begin{enumerate}\itemsep0pt
\item For any map $a$ of the domain of $G$ with rank $k$, the semigroup
$\langle G,a\rangle\setminus G$ contains an idempotent of rank $k$;
\item $G$ has the $k$-universal transversal property.
\end{enumerate}
\end{theorem}

\begin{proof}
Suppose that $G$ has the $k$-ut property, and let $a$ be a map of rank $k$,
with kernel $P$ and image $S$. Choose $g$ such that $Sg$ is a transversal to
the kernel of $a$. Then $ag$ maps $Sg$ to itself; so some power of $ag$ fixes
$Sg$ pointwise, and is an idempotent of rank $k$.

Conversely, let $S$ be a $k$-set and $P$ a $k$-partition. Let $a$ be a map
with kernel $P$ and image $S$. By hypothesis, there is an idempotent
 $e\in\langle G,a\rangle\setminus G$ of rank $k$. Without loss of generality, 
$e=ag_1ag_2\cdots ag_r$. (If the expression for $e$ begins with an element 
of $G$, conjugating by its inverse gives an idempotent of the stated form.)
Now the rank of $ag_1$ is equal to $k$, so $Sg_1$ is a transversal for $P$.\qed
\end{proof}

Thus, the $k$-ut property is necessary for the property of interest to us. It is
obvious that a $k$-homogeneous group has the $k$-ut property. We summarise the results
of \cite{ac_tams} for future reference.

\begin{theorem}\label{th4a}
For $n<11$ and $2\leq k\leq \lfloor\frac{n+1}{2}\rfloor$, a group  $G\leq \sym$ with the $k$-universal transversal property is
$k$-homogeneous, with the following exceptions:
\begin{enumerate}\itemsep0pt
\item $n=5$,  $G\cong C_{5}$ or $D(2*5)$ and $k=2$;
\item $n=6$, $G\cong \psl(2,5)$ and $k=3$;
\item  $n=7$,  $G\cong C_{7}$ or $G\cong D(2*7)$, and $k=2$; or
$G\cong\agl(1,7)$ and $k=3$;
\item $n=8$, $G\cong \pgl(2,7)$ and $k=4$;
\item  $n=9$,  $G\cong 3^{{2}}:4$ or $G\cong 3^{{2}}:D(2*4)$ and $k=2$;
\item $n=10$, $G\cong \mathcal{A}_{5} $ or $G\cong \mathcal{S}_{5}$ and $k=2$; or
$G\cong\psl(2,9)$ or $G\cong \mathcal{S}_{6}$   and $k=3$.
\end{enumerate}
\end{theorem}

\begin{theorem}\label{th4b}
Let $n\geq 11$, $G\leq \sym$.
If $6\leq k\leq \lfloor \frac{n+1}{2}\rfloor$, then the following are equivalent:
\begin{enumerate}\itemsep0pt
\item $G$ has the $k$-universal transversal property;
\item $\alt\leq G$.
\end{enumerate}
\end{theorem}

\begin{theorem}\label{th4c}
Let $n\geq 11$, $G\leq \sym$. The following are equivalent:
\begin{enumerate}\itemsep0pt
\item $G$ has the $5$-universal transversal property;
\item $G$ is 5-homogeneous, or  $n=33$ and $G=\pgaml(2,32)$.
\end{enumerate}
\end{theorem}

\begin{theorem}\label{th4d}
Let $n\geq 11$, $G\leq \sym$ and let $2\leq k\leq \lfloor \frac{n+1}{2}\rfloor$. If $G$ is $4$-homogeneous,  or $n=12$ and $G=\M_{11}$, then $G$ has the $4$-universal transversal property.
If there are further groups possessing the $4$-universal transversal property, then they must be groups $G$ such that $\psl(2,q)\leq G\leq \pgaml(2,q)$, with either $q$ prime  or $q=2^p$ for $p$ prime.
\end{theorem}

\begin{theorem}\label{th4e}
Let $n\geq 11$, $G\leq \sym$ and let $2\leq k\leq \lfloor \frac{n+1}{2}\rfloor$.Then $G$ has the $3$-universal transversal property if $G$ is $3$-homogenous,
or is one of the following groups:
\begin{enumerate}\itemsep0pt
\item $\psl(2,q)\leq G\leq \psigl(2,q)$, where $q\equiv 1 (\mbox{mod }4)$;
\item $\Sp(2d,2)$ with $d\geq 3$, in either of its $2$-transitive representations;
\item $2^{{2d}}:\Sp(2d,2)$;
\item $\HS$;
\item $\Co_{3}$;
\item $2^{{6}}:G_{2}(2)$ and its subgroup of index $2$;
\item $\agl(1,p)$ where, for all $c\in {\GF}(p)\setminus\{0,1\}$, $|\langle -1,c,c-1\rangle|=p-1$.
\end{enumerate}
If there are more groups possessing the $3$-universal transversal property, then they must be Suzuki groups $\Sz(q)$, possibly with field automorphisms adjoined, and/or subgroups of index $2$ in $\agl(1,p)$ for $p\equiv 11$ $(mod\ 12)$.
\end{theorem}

However, the $2$-ut property is equivalent to primitivity, and no more can
be said! This is the case that we will consider in greatest detail in our
investigations here.

By the converse direction in Theorem~\ref{main}, if $G$ is $\sym$ or $\alt$,
then $\langle G,a\rangle\setminus G$ is idempotent-generated for any
non-permuatation $G$. Hence the following theorem is true:

\begin{theorem}
Let $k\ge6$ and $n\ge2k+1$. Then $\langle G,a\rangle\setminus G$ is 
idempotent-generated for every rank $k$ map $a$ if and only if $G=\sym$ or
$G=\alt$.
\end{theorem}

We will extend this result to the cases $k=3,4,5$ in Section~\ref{s:suth}.

\section{The strong universal transversal property}\label{sktu}

In Theorem \ref{utid} we proved that the $k$-id property implies the $k$-ut property. In this section we introduce a new condition, that we call \emph{strong $k$-ut property,} and prove that this new condition (together with $k$-homogeneity) implies the $k$-id property.

Let $k\le n/2$, let $X=\{1,\ldots,n\}$ and let $S_n$ be the symmetric group on $X$. A primitive group $G\le S_n$ is said to have the {\em strong $k$-ut property} if, given any $(k+1)$-tuple $(a_1,\ldots,a_{k+1})\in X^{k+1}$ (of $k+1$ different elements in $X$) and any ordered $k$-partition $P=(A_1,\ldots,A_k)$ of $X$, there exists $g\in G$ such that $\{a_1,\ldots,  a_k\}g$ is a section for $P$ and $a_1g,a_{k+1}g\in A_i$, for some $i\in \{1,\ldots,k\}$.

\begin{theorem}\label{main3}
Let $k\le n/2$ and $G\le S_n$. If $G$ is $k$-homogeneous and possesses the strong $k$-ut property, then for any rank~$k$ map $t$, the semigroup $\langle G,t\rangle\setminus G$ is idempotent-generated.
\end{theorem}
	
This theorem will be proved in a sequence of lemmas. 
We start by introducing some notation. Suppose  $P$ is a partition of a set $X$ and $b\in X$. Then the part in $P$ containing $b$ is represented by $[b]_P$. 

\begin{lemma}\label{5}
Let $G$ be a $k$-homogeneous group and let 
\[
\varepsilon:=\left(
\begin{array}{cccccc}
[a_1,x_1]_P&\ldots&[a_i,x_i]_P&[a_{i+1}]_p&\ldots &[a_k]_P\\
a_1&\ldots&a_i&a_{i+1}&\ldots &a_k
\end{array}
\right)
\]
be an idempotent with image and kernel classes as above. Then there exists in $\langle G,\varepsilon\rangle$ and idempotent of the form
\[\left(
\begin{array}{cccccc}
[a_1,x_1]_P&\ldots&[a_i,x_i]_P&[a_{i+1}]_p&\ldots &[a_k]_P\\
x_1&\ldots&x_i&a_{i+1}&\ldots &a_k
\end{array}
\right).
\]
\end{lemma}
\begin{proof}
Since $G$ is $k$-homogeneous, pick $g\in G$ such that $\{a_1,\ldots,a_i,a_{i+1},\ldots,a_k\}g=\{x_1,\ldots,x_i,a_{i+1},\ldots, a_k\}$. Then 
\[
\varepsilon g=\left(
\begin{array}{cccccc}
[a_1,x_1]_P&\ldots&[a_i,x_i]_P&[a_{i+1}]_p&\ldots &[a_k]_P\\
y_{1\sigma}&\ldots&y_{i\sigma}&y_{(i+1)\sigma}&\ldots &y_{k\sigma}
\end{array}
\right),
\]
where $\{y_1,\ldots,y_k\}=\{x_1,\ldots,x_i,a_{i+1},\ldots, a_k\}$ and $\sigma\in S_k$. Therefore, for some natural $\omega$, we have 
\[
(\varepsilon g)^\omega=\left(
\begin{array}{cccccc}
[a_1,x_1]_P&\ldots&[a_i,x_i]_P&[a_{i+1}]_p&\ldots &[a_k]_P\\
x_1&\ldots&x_i&a_{i+1}&\ldots& a_k
\end{array}
\right),
\]
as required.\qed
\end{proof}

 Suppose $Q$ is a partition of $X$ and $B$ is a part in $Q$; suppose that  $ag\in B$, for some $a\in X$ and $g\in G$. Then $a\in Bg^{-1}$ and hence $Bg^{-1}=[a]_{Pg^{-1}}$, the part containing $a$ in the partition $Pg^{-1}$. In short, if $P:=(B,\ldots)$ is a partition and $ag\in B$, then $Pg^{-1}=([a]_{Pg^{-1}},\ldots)$.

\begin{lemma}\label{6}
Let $G$ be a $k$-homogeneous group possessing the strong $k$-ut property, and let 
\[
\varepsilon:=\left(
\begin{array}{cccccc}
[a_1,c]_P&[a_2]_P&[a_3]_P&\ldots & [a_k]_P\\
   a_1      & a_2     &a_3      &\ldots  & a_k
\end{array}
\right)
\]
be an idempotent with image and kernel classes as above. Then there exists in $\langle G,\varepsilon\rangle$ and idempotent of the form
\[
\left(
\begin{array}{cccccc}
[a_1,a_2]_Q&[c]_Q&[a_3]_Q&\ldots & [a_k]_Q\\
   a_1      & c     &a_3&\ldots  & a_k
\end{array}
\right),
\]
for some partition $Q$.
\end{lemma}
\begin{proof}
By the strong $k$-ut property, given the tuple $(a_1,c,a_3\ldots,a_k,a_2)$, there exists $g\in G$ such that $\{a_1,c,a_3\ldots, a_k\}g$ is a section for $P$ and $a_1g,a_2g\in [a_i]_P$; therefore, $Q:=Pg^{-1}=([a_1,a_2]_Q,[c]_Q,[a_3]_Q,\ldots,[a_k]_Q)$.  Thus 
\[
g\varepsilon =\left(
\begin{array}{cccccc}
[a_1,a_2]_{Q}&[c]_{Q}&[a_3]_{Q}&\ldots &[a_k]_{Q}\\
y_{1}&y_{2}&y_{3}&\ldots &y_{k}
\end{array}
\right),
\]
where $\{y_1,\ldots,y_k\}=\{a_1,\ldots, a_k\}$. Therefore, for some $h\in G$, we have $\{y_1,\ldots,y_k\}h=\{a_1,c,a_3,\ldots, a_k\}$ and hence, for some natural $\omega$, we have 
\[
(g\varepsilon h)^\omega=\left(
\begin{array}{cccccc}
[a_1,a_2]_{Q}&[c]_{Q}&[a_3g]_{Q}&\ldots &[a_kg]_{Q}\\
a_{1}                      & c                  &a_3                  &\ldots &a_k
\end{array}
\right),
\]
as required.\qed
\end{proof}
	
	If, in the end of the proof above, we pick $h'\in G$ such  that $\{y_1,\ldots,y_k\}h=\{c,a_2,a_3,\ldots, a_k\}$, then $(g\varepsilon h')^\omega$ will have the form 
\[
(g\varepsilon h)^\omega=\left(
\begin{array}{cccccc}
[a_1,a_2]_{Q}&[c]_{Q}&[a_3g]_{Q}&\ldots &[a_kg]_{Q}\\
a_{2}                      & c                  &a_3                  &\ldots &a_k
\end{array}
\right).
\]
Thus we immediately get the following.

\begin{lemma}\label{7}
Let $G$ be a $k$-homogeneous group possessing the strong $k$-ut property, and let 
\[
\varepsilon:=\left(
\begin{array}{cccccc}
[a_1,c]_P&[a_2]_P&[a_3]_P&\ldots & [a_k]_P\\
   a_1      & a_2     &a_3      &\ldots  & a_k
\end{array}
\right)
\]
be an idempotent with image and kernel classes as above. Then there exists in $\langle G,\varepsilon\rangle$ and idempotent of the form
\[
\left(
\begin{array}{cccccc}
[a_1,a_2]_Q&[c]_Q&[a_3]_Q&\ldots & [a_k]_Q\\
   a_2      & c     &a_3&\ldots  & a_k
\end{array}
\right),
\]
for some partition $Q$.
\end{lemma}

\begin{lemma}\label{8}
Let $G$ be a $k$-homogeneous group possessing the strong $k$-ut property, and let 
\[
\varepsilon:=\left(
\begin{array}{cccccc}
[a_1]_P&[a_2]_P&[a_3]_P&\ldots & [a_k]_P\\
   a_1      & a_2     &a_3      &\ldots  & a_k
\end{array}
\right)
\]
be an idempotent with image and kernel classes as above. Then there exists in $\langle G,\varepsilon\rangle$ a map of the form
\[
\left(
\begin{array}{cccccc}
[a_1,c]_Q&[a_2]_Q&[a_3]_Q&\ldots & [a_k]_Q\\
   a_1      & a_2   &a_3&\ldots  & a_k
\end{array}
\right),
\]
for some partition $Q$.
\end{lemma}

\begin{proof}
Suppose that $|[a_1]_P|>1$; then the result follows. Otherwise, since $k\le n/2$, there must be one $a_i\in Xt$ such that $c\in [a_i]_P\setminus \{a_i\}$. Thus, for the tuple $(a_1,a_2,\ldots,a_k,c)$, there exists $g\in G$ such that $a_1g,cg\in [a_j]_P$ (for some $j$) and $\{a_1,\ldots,a_k\}g$ is a section for $P$. Therefore, for $Q:=Pg^{-1}$, we have $Q=([a_1,c],[a_2],\ldots,[a_k])$. Thus,
\[
g^{-1}\varepsilon g=\left(
\begin{array}{cccccc}
[a_1,c]_Q&[a_2]_Q&[a_3]_Q&\ldots & [a_k]_Q\\
   a_1      & a_2   &a_3&\ldots  & a_k
\end{array}
\right),
\] as required.\qed
\end{proof}

\begin{lemma}
Let $G$ be a $k$-homogeneous group possessing the strong $k$-ut property, and let 
\[
\varepsilon :=\left(
\begin{array}{cccccc}
[a_1]_P&[a_2]_P&[a_3]_P&\ldots & [a_k]_P\\
   a_1      & a_2     &a_3      &\ldots  & a_k
\end{array}
\right)
\] 
be an idempotent with image and kernel classes as above. Then the following map 
\[
\left(
\begin{array}{cccccc}
[a_1]_P&[a_2]_P&[a_3]_P&\ldots & [a_k]_P\\
   a_2      & a_1   &a_3&\ldots  & a_k
\end{array}
\right)
\]
can be written as a product of idempotents in  $\langle G,\varepsilon\rangle$.
\end{lemma}

\begin{proof}	
By Lemma \ref{8}, there exists an idempotent map 
\[\varepsilon_1=
\left(
\begin{array}{cccccc}
[a_1,c]_Q&[a_2]_Q&[a_3]_Q&\ldots & [a_k]_Q\\
   a_1      & a_2   &a_3&\ldots  & a_k
\end{array}
\right)\in \langle G,\varepsilon\rangle.
\]

By Lemma \ref{5} there exists an idempotent map 
\[\varepsilon_2=
\left(
\begin{array}{cccccc}
[a_1,c]_Q&[a_2]_Q&[a_3]_Q&\ldots & [a_k]_Q\\
   c      & a_2   &a_3&\ldots  & a_k
\end{array}
\right)\in \langle G,\varepsilon_1\rangle\subseteq \langle G,\varepsilon\rangle.
\]

By Lemma \ref{6} there exists an idempotent map 
\[\varepsilon_3=
\left(
\begin{array}{cccccc}
[a_1,a_2]_Q&[c]_Q&[a_3]_Q&\ldots & [a_k]_Q\\
   a_1     & c   &a_3&\ldots  & a_k
\end{array}
\right)\in \langle G,\varepsilon_1\rangle\subseteq \langle G,\varepsilon\rangle.
\]

By Lemma \ref{8}, there exists an idempotent map 
\[\varepsilon_4=
\left(
\begin{array}{cccccc}
[a_2,c]_R&[a_1]_R&[a_3]_R&\ldots & [a_k]_R\\
   a_2      & a_1   &a_3&\ldots  & a_k
\end{array}
\right)\in \langle G,\varepsilon\rangle.
\]

Now $\varepsilon_2\varepsilon_3\varepsilon_4$ gives the mapping 
\[
\left(
\begin{array}{cccccc}
[a_1,c]_Q&[a_2]_Q&[a_3]_Q&\ldots & [a_k]_Q\\
   a_2      & a_1   &a_3&\ldots  & a_k
\end{array}
\right)\in \langle G,\varepsilon\rangle.
\]
Finally,  $\varepsilon \varepsilon_2\varepsilon_3\varepsilon_4$ gives 
\[
\left(
\begin{array}{cccccc}
[a_1]_P&[a_2]_P&[a_3]_P&\ldots & [a_k]_P\\
   a_2      & a_1   &a_3&\ldots  & a_k
\end{array}
\right)\in \langle G,\varepsilon\rangle,
\]as required.\qed
\end{proof}

\begin{lemma}
Let $G$ be a $k$-homogeneous group possessing the strong $k$-ut property, and let 
\[
q :=\left(
\begin{array}{cccccc}
[a_1]_P&[a_2]_P&[a_3]_P&\ldots & [a_k]_P\\
   a_{1\sigma}      & a_{2\sigma}     &a_{3\sigma}      &\ldots  & a_{k\sigma}
\end{array}
\right)
\] 
be a map with image and kernel classes as above. Let  
\[
\varepsilon=\left(
\begin{array}{cccccc}
[a_1]_P&[a_2]_P&[a_3]_P&\ldots & [a_k]_P\\
   a_1      & a_2   &a_3&\ldots  & a_k
\end{array}
\right)
\]
and $x,y\in \{a_1,\ldots,a_k\}$. Then the map $q(x\ y)$ can be written as $qu$, where $u$ is a product of idempotents in  $\langle G,\varepsilon\rangle$.
\end{lemma}

\begin{proof}
Without loss of generality we can assume that $x=a_1$ and $y=a_2$. Then, by the previous lemma, we know that the map 
\[
\bar{\varepsilon}=\left(
\begin{array}{cccccc}
[a_1]_P&[a_2]_P&[a_3]_P&\ldots & [a_k]_P\\
   a_2      & a_1   &a_3&\ldots  & a_k
\end{array}
\right)
\]
can be generated by the idempotents in $\langle G,\varepsilon\rangle$. Since it is clear that $q(a_1\ a_2)=q\bar{\varepsilon}$, the result follows. \qed
\end{proof}

The next theorem shows that if $t$ is a rank $k$ map contained in a maximal subgroup of $T_n$, then $t$ is generated by the idempotents of 
$\langle G,t\rangle$, for any $k$-homogeneous group $G$ possessing the $k$-ut property. 

\begin{theorem}\label{11a} Let $G\le S_n$ be a $k$-homogeneous group possessing the strong $k$-ut property. 
Let $t\in T_n$ be a rank $k$ map such that $Xt$ is a section for the kernel of $t$. Then if $q$ has the same kernel and the same image as $t$, then $q$ is generated by the idempotents in $\langle G,t\rangle$. 
\end{theorem}

\begin{proof}
Let $t$ be as in the assumptions. Then, for some natural $\omega$, we have 
\[\varepsilon:=t^\omega=
\left(
\begin{array}{cccccc}
[a_1]_P&[a_2]_P&[a_3]_P&\ldots & [a_k]_P\\
   a_1      & a_2   &a_3&\ldots  & a_k
\end{array}
\right),
\]and any map $q$ having the same kernel and image as $t$ is of the form 
\[q=
\left(
\begin{array}{cccccc}
[a_1]_P&[a_2]_P&[a_3]_P&\ldots & [a_k]_P\\
   a_{1\sigma}      & a_{2\sigma}   &a_{3\sigma}&\ldots  & a_{k\sigma}
\end{array}
\right).
\]
Therefore $q=\varepsilon(x_1\ y_1)\ldots (x_m\ y_m)$, where $\{x_1,\ldots,x_m,y_1,\ldots,y_m\}=\{a_1,\ldots,a_k\}$. By the previous result, each transposition can be replaced with a product of idempotents in $\langle G,\varepsilon \rangle\subseteq \langle G,t \rangle$ and hence the result follows.\qed
\end{proof}
	
The next sequence of results aims at dealing with maps that do not belong to a maximal subgroup of $T_n$. 
	
\begin{lemma}
Let $G$ be a $k$-homogeneous group possessing the strong $k$-ut property. Let $t\in T_n$ be a rank $k$ map and let $Y\subseteq X$ be a $k$-set such that $|Xt\setminus Y|=1$. Then there exists $u\in T_n$ such that $Xtu=Y$ and $u$ can be written as a product of idempotents in $\langle G,t\rangle$. 
\end{lemma}

\begin{proof}
Let $t$ be as in the assumptions. Since $G$ is $k$-homogeneous, there exists an idempotent having the same image as $t$:
\[\varepsilon=
\left(
\begin{array}{cccccc}
[a_1]_P&[a_2]_P&[a_3]_P&\ldots & [a_k]_P\\
   a_1      & a_2   &a_3&\ldots  & a_k
\end{array}
\right)\in \langle G,t\rangle.
\]
Suppose $Y=\{y,a_2,\ldots,a_k\}$. Then, by Lemma \ref{8}, there exists in $\langle G,\varepsilon\rangle \subseteq \langle G,t\rangle$ an idempotent of the form 
\[\overline{\varepsilon}=
\left(
\begin{array}{cccccc}
[a_1,y]_Q&[a_2]_Q&[a_3]_Q&\ldots & [a_k]_Q\\
   a_1      & a_2   &a_3&\ldots  & a_k
\end{array}
\right)
\]
and hence, by Lemma \ref{5}, there exists in $\langle G,\varepsilon\rangle \subseteq \langle G,t\rangle$ an idempotent of the form 
\[\overline{\overline{\varepsilon}}=
\left(
\begin{array}{cccccc}
[a_1,y]_Q&[a_2]_Q&[a_3]_Q&\ldots & [a_k]_Q\\
   y      & a_2   &a_3&\ldots  & a_k
\end{array}
\right).
\] 
It is clear that $Xt\overline{\overline{\varepsilon}}=Y$, as desired.\qed
\end{proof}

\begin{lemma}
Let $G$ be a $k$-homogeneous group possessing the strong $k$-ut property. Let $t\in T_n$ be a rank $k$ map. Then there exists $u\in \langle G,t\rangle$ such that $\ker(u)=\ker(t)$, $Xu=Xt$, and $u$ is generated by idempotents in $\langle G,t\rangle$. 
\end{lemma}

\begin{proof}	
%
Let $Xu =\{e_1,\ldots,e_k\}$ and $Xt =\{a_1,\ldots,a_k\}$ such that $\ker(u)=\ker(t)$, $u$ is a product of idempotents in $\langle G,t\rangle$, and $Xu\cap Xt$ is maximal (that is, no other $v$ satisfying $\ker(v)=\ker(t)$ and $v$ is a product of idempotents in $\langle G,t\rangle$, has a larger intersection with $Xt$). Then we claim that $Xu=Xt$. In fact, suppose  $Xu =\{a_1,\ldots,a_i,e_{i+1},\ldots,e_k\}$ and $Xt =\{a_1,\ldots,a_k\}$.  
Then, by the previous lemma, there exists $v\in T_n$ such that $X u v= \{a_1,\ldots,a_i,a_{i+1},e_{i+2},\ldots, e_k\}$, where $v$ is a product of idempotents in $\langle G,u\rangle\subseteq \langle G,t\rangle$, thus contradicting the maximality of $u$. \qed
\end{proof}
 
 \begin{lemma}
Let $G$ be a $k$-homogeneous group possessing the strong $k$-ut property. Let $t\in T_n$ be a rank $k$ map and $\gamma\in G$. Then $t\gamma$ can be written as a product of idempotents in $ \langle G,t\rangle$. 
\end{lemma}

\begin{proof}	
Since $G$ is $k$-homogeneous, it follows that there exists $g\in G$ such that $Xt\gamma g$ is a section for $\ker(t\gamma)$ and hence $\varepsilon:=(gt\gamma)^\omega$ is idempotent (for some natural $\omega$), $X\varepsilon =Xt\gamma$ and $\varepsilon\in \langle G,t\gamma\rangle=\langle G,t\rangle$. Denote by $H_\varepsilon$ the maps in $T_n$ that have the same image and the same kernel as $\varepsilon$. By Theorem \ref{11a}, every element in $H_\varepsilon$ is generated by the idempotents in $\langle G,\varepsilon\rangle\subseteq \langle G,t\rangle$.
 
By the previous result there exists a map $u\in  \langle G,t\rangle$ such that $\ker(u)=\ker(t\gamma)$, $Xu=Xt\gamma$ and $u$ is a product of idempotents in $\langle G,t\gamma\rangle(=\langle G,t\rangle)$. As there exists $p\in H_\varepsilon$ such that $up=t\gamma$, it follows that $t\gamma$ is generated by the idempotents in $\langle G,t\rangle$. \qed
\end{proof}

We have all we need to prove Theorem \ref{main3}.

\begin{proof}
The result holds if $gth$ is generated by the idempotents in $\langle G,t\rangle\setminus G$, for all $g,h\in G$. By the previous theorem $thg$ is generated by the idempotents in $\langle G,t\rangle\setminus G$; therefore $gth=g(thg)g^{-1}$ is the conjugate of a product of idempotents in $\langle G,t\rangle$ and hence it is a product of idempotents in the same semigroup. 
\qed
\end{proof}

\section{Strong ut property and homogeneity}\label{sutph}
\label{s:suth}

In this section we investigate the relationship between the strong universal
transversal property of a permutation group and homogeity (or transitivity)
of the group. 

We recall the property here. $G$ has the strong $k$-ut property if, given any
$(k+1)$-tuple $(a_1,\ldots,a_{k+1})$ of distinct points and any $k$-partition
$P=(A_1,\ldots,A_k)$, there exists $g\in G$ such that $\{a_1,\ldots,a_k\}g$
is a section for $P$ and $a_1g,a_{k+1}g$ lie in the same part.

We begin with the following observation. Unlike our earlier investigation
of the $k$-ut property, the proof is elementary, and does not require the
concept of $(k-1,k)$-homogeneity or the classification of $k$-homogeneous
groups.

\begin{prop}
A permutation group with the strong $k$-ut property is $(k-1)$-homogeneous.
\end{prop}

\begin{proof}
Suppose that $G$ has the strong $k$-ut property. Let $\{x_1,\ldots,x_{k-1}\}$
and $\{y_1,\ldots,y_{k-1}\}$ be two $k$-subsets of $X$. Let $(a_1,\ldots,
a_{k+1})$ be a $(k+1)$-tuple of distinct points satisfying $a_{i+1}=x_i$
for $i=1,\ldots,k$ (the points $a_1$ and $a_{k+1}$ are arbitrary); and let
$P$ be the partition whose parts are $\{y_i\}$ for $i=1,\ldots,k-1$ and
$X\setminus\{y_1,\ldots,y_{k-1}\}$. Choose $g$ as in the strong $k$-ut
property. Then $a_1g$ and $a_{k+1}g$ are in the same part of $P$, necessarily
the last one (since the others are singletons); so $\{a_2g,\ldots,a_kg\}$ is
a section for the remaining parts. This means that $g$ maps
$\{x_1,\ldots,x_{k-1}\}$ to $\{y_1,\ldots,y_{k-1}\}$, as required.\qed
\end{proof}

In the other direction, we have the following.

\begin{prop}
Let $G$ be a permutation group of degree greater than $k$, where $k>1$. Suppose
that the setwise stabiliser of any $(k+1)$-set induces a $2$-homogeneous group
of permutations of this set. Then $G$ has the strong $k$-ut property.
In particular, a $(k+1)$-transitive or generously $k$-transitive group has the
strong $k$-ut property.
\label{p:k+1-stab}
\end{prop}

(A permutation group is \emph{generously $k$-transitive} if the
setwise stabiliser of any $(k+1)$-set acts on it as the symmetric group.)

\begin{proof}
Suppose that $G$ satisfies the hypothesis. We observe first that $G$ is
$k$-homogeneous. For certainly the stabiliser of a $(k+1)$-set acts
transitively on it; so if two $k$-sets intersect in $k-1$ points, then we can
map one to the other by a permutation fixing their union. Now any two $k$-sets
can be connected by a chain of $k$-sets in which successive members meet in
$k-1$ points.

Now suppose that $(a_1,\ldots,a_{k+1})$ is a $(k+1)$-tuple of distinct points,
and $P=(A_1,\ldots,A_k)$ a $k$-partition. By $k$-homogeneity, we can find
$g_1\in G$ such that $\{a_1,\ldots,a_k\}g_1$ is a section for $P$. Suppose
that $a_{k+1}g_1$ lies in the same part of $P$ as $a_ig_1$. By hypothesis,
we can find a permutation $g_2\in G$ fixing $\{a_1,\ldots,a_{k+1}\}$ and
mapping $\{a_1,a_{k+1}\}$ to $\{a_i,a_{k+1}\}$. Then $g=g_2g_1$ is the element
required by the $k$-ut property.\qed
\end{proof}

We can now state our results on the $k$-id property for $k\ge4$ and groups of
degree at least $11$.

\begin{theorem}
Let $G$ be a permutation group of degree $n$.
\begin{enumerate}
\item Suppose that $k\ge6$ and $n\ge2k+1$. Then $G$ has the $k$-id property
if and only if $G$ is $S_n$ or $A_n$.
\item Suppose that $n\ge11$. Then $G$ has the $5$-id property if and only if
$G$ is $S_n$, $A_n$, $M_{12}$ (with $n=12$) or $M_{24}$ (with $n=24$),
or possibly $\pgaml(2,32)$ (with $n=33$).
\item Suppose that $n\ge11$. Then $G$ has the $4$-id property if and only if
$G$ is $S_n$, $A_n$, $M_n$ (with $n=11,12,23,24$), or possibly 
$\psl(2,q)\le G\le\pgaml(2,q)$ with either $q$ prime congruent to
$11\pmod{12}$ or $q=2^p$ with $p$ prime, or $G=M_{11}$ with $n=12$.
\end{enumerate}
\end{theorem}

\begin{proof}
The fact that no other groups can arise follows from the results of
\cite{ac_tams}, and the fact that the only $4$-homogeneous groups of degree
at least $11$ are symmetric, alternating and Mathieu groups and
$\pgaml(2,32)$.

Both the groups $M_{12}$ and $M_{24}$ are $5$-transitive and have two orbits
on $6$-sets, the stabiliser of a $6$-set in either orbit acting
$2$-transitively on it. Indeed, $M_{24}$ is generously $5$-transitive.
It follows that both these groups $G$ have the strong
$5$-ut property, and hence that $\langle G,a\rangle\setminus G$ is
idempotent-generated for any map $a$ of rank $5$.

Similarly, the groups $M_{11}$ and $M_{23}$ have the property that the
stabiliser of a $5$-set acts $2$-transitively on it; so these groups have
the $4$-id property.\qed
\end{proof}

We have not completed the analysis of possible examples with
$\psl(2,q)\le G\le\pgaml(2,q)$, or for $M_{11}$ (degree~$12$).

\subsection{The case $k=3$}
For groups with the $3$-id property, we know that they must be among those
with the $3$-ut property described in Theorem~\ref{th4e}. We have made
little progress in deciding about these groups. One observation is:

\begin{prop}
Let $p$ be an odd prime. Then $\agl(2,p)$ does not have the strong $3$-ut
property.
\end{prop}

\begin{proof}
We can assume that $p\equiv11\pmod{12}$, since in other cases the group does
not have the $3$-ut property. In particular, $p\equiv3\pmod{4}$, so $-1$ is
a quadratic non-residue mod~$p$.

Let $R$ and $N$ be the sets of quadratic residues and non-residues. It is not
the case that $R=\{1,2,\ldots,(p-1)/2\}$, since then $-1=2\cdot(p-1)/2$ would
be a residue. So there is an element $c$ such that $c-1\in N$ and $c\in R$.

Let $S=(-1,0,c-1,c)$ and $P=\{\{0\},R,N\}$. We claim that the strong $3$-ut
property fails for these elements. Suppose that $g\in\agl(2,p)$ maps
$(-1,0,c-1)$ to a section for $P$ and that $(-1)g$ and $cg$ belong to the same
part. By symmetry, we can assume that $0g=0$, so that $g$ is multiplication by
$\lambda$, say. Then $(-1)g$ and $cg$ have the same quadratic character, and
$(c-1)g$ the opposite. But this is impossible, since $-1$ and $c-1$ are
non-residues and $c$ is a residue.\qed
\end{proof}

We have also checked with GAP that the Higman--Sims group does not have the strong $3$-ut.
Whether it has the $3$-id property is not known.

On the other hand, Proposition~\ref{p:k+1-stab} shows that $\M_{22}$ and its
automorphism group, and $M_{11}$ (degree~$12$) have the strong $3$-ut property,
and hence the $3$-id property. In addition, all $4$-transitive groups
(symmetric, alternating and Mathieu groups) have the $3$-id property.

\section{The case $n<11$}\label{11s}

In this section we are going to handle the groups of small degree. 

\begin{theorem} Let $G\le S_n$ be a group, $n<11$ and $k\le n/2$. 
Then $G$ has the strong $k$-ut property if and only if $G$ contains the alternating group and $k<n$, or  $k\le n/2$ and one of the following holds:
\begin{enumerate}
\item $n=5$ and $G\cong \agl(1,5)$ for $k=2$;
\item $n=6$ and $G\cong \psl(2,5)$ or $\pgl(2,5)$ for $k=2$;
\item $n=7$ and $G\cong 7:3$, $\agl(1,7)$ or $\mbox{L}(3,2)$ for $k=2$;
\item $n=8$ and $G\cong \agl(1,8)$, $\agaml(1,8)$, $\asl(3,2)$, $\psl(2,7)$ or $\pgl(2,7)$ for $k=2$;
\item $n=8$ and $G\cong \agaml(1,8)$, $\asl(3,2)$, $\psl(2,7)$ or $\pgl(2,7)$ for $k=3$;
\item $n=9$ and $G\cong \M_9$, $\agl(1,9)$, $\agaml(1,9)$, $3^2:(2'A_4)$, $\agl(2,3)$, $\psl(2,8)$ or $\pgaml(2,8)$ for $k=2$;
\item $n=9$ and $G\cong \psl(2,8)$ for $k=3$ or $\pgaml(2,8)$ for $k=3,4$;
\item $n=10$ and $G\cong \psl(2,9)$, $\pgl(2,9)$, $\psigmal(2,9)$, $\M_{10}$ or $\pgaml(2,9)$ for $k=2$;
\item $n=10$ and $G\cong \pgl(2,9)$, $\M_{10}$ or $\pgaml(2,9)$ for $k=3$;
\end{enumerate}
\end{theorem}
\begin{proof}
GAP checks all these claims in less than one minute.\qed
\end{proof}

As the groups in the previous theorem have the strong $k$-ut property it follows that all of them have the $k$-id property. Of course among the groups not possessing the strong $k$-ut property there might be some satisfying the $k$-id. As having the $k$-ut property is a necessary condition for having the $k$-id we just need to go through the list of groups in Theorem \ref{th4a}.

\begin{theorem}
Let $n<11$ and $G\le S_n$. Let $k<n/2$. Then $G$ has the $k$-id property if and only if $G$ and $k$ are listed in the previous theorem, or 
\begin{enumerate}
\item $n=5$,  $G\cong C_{5}$ or $D(2*5)$ and $k=2$;
\item $n=6$ and $G\cong \psl(2,5)$ or $\pgl(2,5)$ for $k=3$;
\item $n=7$ and $G\cong C_7$ or $D(2*7)$ for $k=2$; 
\item $n=8$ and $G\cong \agl(1,8)$ for $k=3$; 
\item $n=10$ and $G\cong A_5$ or $S_5$ for $k=2$; 
\item $n=10$ and $G\cong \psl(2,9)$ or $\psigmal(2,9)$ for $k=3$.
\end{enumerate}

\end{theorem}
\begin{proof}
To check the $k$-id property we used the functions in \cite{ArMiSc} (available in its companion website). The following table provides for the relevant groups $G$ one rank $k$ map $t$ such that $\langle G,t\rangle \setminus G$ is not idempotent generated.     
\begin{longtable}{|cccc|}
\hline
group&$n$&$k$&$t$\\ \hline \hline
$C_7$&$7$&$3$&$2211552$ \\ \hline
$D(2*5)$&$7$&$3$&$3344663$ \\ \hline
$7:3$&$7$&$3$&$3114433$ \\ \hline
$\agl(1,7)$&$7$&$3$&$1334441$ \\ \hline
$L(3,2)$&$7$&$3$&$1244411$ \\ \hline
$\agl(1,8)$&$8$&$4$&$12433111$ \\ \hline
$\agaml(1,8)$&$8$&$4$&$12433111$ \\ \hline
$\asl(3,2)$&$8$&$4$&$12341111$ \\ \hline
$\psl(2,7)$&$8$&$4$&$12255771$ \\ \hline
$\pgl(2,7)$&$8$&$4$&$12335511$ \\ \hline
$3^2:4$&$9$&$2$&$111777777$ \\ \hline
$3^2:D(2*4)$&$9$&$2$&$111777777$ \\ \hline
$M_9$&$9$&$3$&$114477771$ \\ \hline
$\agl(1,9)$&$9$&$3$&$114477771$ \\ \hline
$\agaml(1,9)$&$9$&$3$&$114477771$ \\ \hline
$3^2:(2'A_4)$&$9$&$3$&$114477771$ \\ \hline
$\agl(2,3)$&$9$&$3$&$114477771$ \\ \hline
$A_5$&$10$&$3$&$10\  10\  1155555\  10$ \\ \hline
$S_5$&$10$&$3$&$10\  10\  1199999\  10$ \\ \hline
\end{longtable}\qed
\end{proof}

\section{The case $k=2$: the cornerstone}\label{corners}

In the previous sections we worked on the classification of groups with the $k$-id property, for $k>2$; we are now going to handle the classification of permutation groups with the $2$-id property. As said above this is the most interesting and demanding case.

Let $G\le S_n$ be a permutation group, let $k\le n$, let $P$ be a $k$-partition of $\{1,\ldots,n\}$ and let $S$ be a $k$-set contained in $\{1,\ldots,n\}$. 
Recall that the \emph{Houghton graph} $H(G,k,P,S)$ is defined as follows: the vertex
set of this bipartite graph is the union of the $G$-orbits containing 
$P$ and $S$, and $P'$ is joined to $S'$ if $S'$ is a transversal for $P'$.

The cornerstone of our result for $k=2$ is the following result which is the main theorem of this section.

\begin{theorem}
Let $G\le S_n$ be a permutation group of $X=\{1,\ldots,n\}$, and let $t$ be a rank $2$ map. Then the following are equivalent:
\begin{enumerate}
\item $\langle G,t\rangle\setminus G$ is idempotent generated:
\item $H(G,2,\ker(t),Xt)$ is connected. 
\end{enumerate} 
\end{theorem}
The proof of this result requires some background.

Let $G$ be a group and $0$ a symbol not in $G$. We can extend the multiplication in $G$ to a \emph{group with zero} whose universe is $G\cup \{0\}$, and the multiplication is defined by $x*0=0*x=0$, for all $x\in G\cup\{0\}$. We will denote this new semigroup by $G^0$. 

Let $G$ be a group, $G^0$ the corresponding group with zero, let $I, \Lambda$ be non-empty index sets and $P=(p_{\lambda i})$ a matrix with entries in $G^0$ and \emph{regular} (meaning that every row and every column has at least one non-zero entry). Then we can define a new semigroup $S:=M^0[G;I,\Lambda;P]$, called the \emph{Rees matrix semigroup over $G^0$ with sandwich matrix $P$}, whose universe is $(I\times G\times \Lambda)\cup \{0\}$ and multiplication defined by 
\[\begin{array}{l}
(i,g,\lambda)(j,h,\mu)=\left\{ 
\begin{array}{ll}
(i,gp_{\lambda j}h,\mu) &\mbox{if $p_{\lambda j}\neq 0$}\\
0&\mbox{otherwise}	
\end{array}\right.\\ \\
(i,g,\lambda)0=0(i,g,\lambda)=0.
\end{array}
\] 
All finite semigroups can be seen as unions of null semigroups (semigroups with zero satisfying the identity $xy=0$) and Rees matrix semigroups over $G^0$;  these two types of semigroups  describe the \emph{local structure} of semigroups and that is why they are so important. 

For a familiar illustration, let $X$ be a finite set and let  $T(X)$ be the monoid of all transformations on $X$; let $k$ be a natural number such that $1\le k\le n$, and let $T_k(X)$ be all the rank $k$ transformations in $T(X)$,  together with an extra symbol denoted by $0$. Given $t,q\in T_k(X)\setminus \{0\}$ we can define a new product as follows:  

\[\begin{array}{l}
t*q=\left\{ 
\begin{array}{ll}
tq &\mbox{if rank}(tq)=k\\
0&\mbox{otherwise}	
\end{array}\right.\\ \\
t*0=0*t=0.
\end{array}
\] With this product, $T_k(X)$ is a semigroup encoding much information about the rank $k$ maps; for example, as a semigroup of transformations, they are generated by rank $k$ idempotents if and only if the semigroup $(T_k(X),*)$ is idempotent generated. The importance of this new product $*$ is that the semigroup $(T_k(X),*)$  is isomorphic to a  Rees matrix semigroup $M^0[G;I,\Lambda;P]$ whose ingredients are:   
\begin{itemize}
\item $\Lambda$, the set of all $k$-subsets contained in $X$; 
\item $I$, the set of all $k$-partitions on $X$; 
\item for $i\in I$ and $\lambda \in \Lambda$ such that $\lambda$ is a transversal for $i$, let $G_{\lambda i}:=\{f\in T_k(X)\mid Xf=\lambda \mbox{ and  ker}(f)=i\}$. It can be proved that all these $G_{\lambda i}$ are isomorphic groups and hence $G$ is taken to be one of them. In the case of maps of rank $k$, all the maps that have kernel $i$ and  image $\lambda$ (where $\lambda$ is a transversal for $i$) form a group of transformations isomorphic to the symmetric group $S_k$; 
\item finally, for the matrix $P$ we have that $p_{\lambda i}\neq0$ ($\lambda\in \Lambda$, $i\in I$) if and only if $\lambda$ is a transversal for $i$.
\end{itemize}
Therefore $(T_k(X),*)$ is isomorphic to $$M^0[S_k;\mbox{$k$-partitions of $X$},\mbox{$k$-subsets of $X$};P],$$ where only the precise value of the $P$ entries is not given; this is because it depends on some free choices and hence a given semigroup can be isomorphic to Rees matrix semigroups with different matrices $P$. Fortunately,  Graham \cite[Theorem 2]{graham} found a normal form for these matrices that we now introduce; see also \cite{gray0,gray1}.

\begin{theorem}[Graham normal form] \label{gform}
Let $S=M^0[G;I,\Lambda;P]$ be a finite Rees matrix semigroup. It is always possible to normalize the structure matrix $P$ to obtain a matrix $Q$ with the following properties: 
\begin{itemize}
\item the matrix $Q$ is a direct sum of $r$ blocks $C_1,\ldots ,C_r$ as suggested in the following picture: 
\[
\begin{blockarray}{cccccc}
 & B_1 & B_2 & \ldots & B_r \\
\begin{block}{c(ccccc)}
 A_1 &C_1 &  &  & 0   \\
 A_1 &  &C_2 &  &    \\
\vdots&   &  &\ddots  &   \\
 A_r &0  &  &  & C_r   \\
\end{block}
\end{blockarray}\ .
 \]
 \item Each matrix $C_i:A_i\times B_i\rightarrow G^0$ is regular and the semigroup generated by the idempotents of $S$ is 
 \[
 \bigcup^r_{i=1}M^0[G_i;A_i,B_i;C_i],
 \]where $G_i$ is the subgroup of $G$ generated by the non-zero entries of $C_i$, for $i=1,\ldots, r$. 
 \item $M^0[G;I,\Lambda;P]$ and $M^0[G;I,\Lambda;Q]$ are isomorphic. 
 \end{itemize}
\end{theorem}

 Let $S=M^0[G;I,\Lambda;P]$ be a Rees matrix semigroup.   Given $C\subseteq I\times \Lambda$, denote by $\Gamma(C)$ the undirected graph with set of vertices $C$ and two vertices $(i,\lambda)$ and $(j,\mu)$ form an edge if and only if $i=j$ or $\lambda=\mu$. An especially relevant subset of $I\times \Lambda$ is 
 \[
 \mathbb{H}_S=\{(i,\lambda)\in I\times \Lambda \mid p_{\lambda i}\neq 0\}. 
 \]
In our $T_k(X)$ example above this is the set of pairs $(i,\lambda)$  such that $\lambda$ is a transversal for the partition $i$. 
 
 The semigroup $S=M^0[G;I,\Lambda;P]$ is said to be connected if  $\Gamma(\mathbb{H}_S)$ is connected.

 \begin{theorem}(\cite[Theorem 3.1]{gray0})\label{gray}
 Let $S=M^0[G;I,\Lambda;P]$ be a finite Rees matrix semigroup in Graham normal form. Then $S$ is idempotent generated if and only if $S$ is connected and the group $G$ is generated by the entries in the matrix $P$.
 \end{theorem}

The next result describes the Rees matrix semigroups $S=M^0[G;I,\Lambda;P]$ in which every entry in $P$ is either $0$ or the identity of $G$. Before stating the theorem we need to introduce a concept. A \emph{polygonal line} in the Cayley table of a semigroup is a sequence of entries in the Cayley table that can be reached by a sequence of chess-rook moves.
For example, suppose we have a semigroup with elements $\{a,b,c,x,w,y,z,\ldots\}$ and multiplication $\circ$, and part of its Cayley table looks as follows: 

\xymatrix{
       \circ& x  &w  & y &\ldots  & z\\
      a& ax&aw&  ay &           &az \\
      b& bx &bw  & by   &&bz \\
    }

$\ $

On this Cayley table we can define the following polygonal line starting on $ax$; this line is said to be \emph{closed} as the initial and terminal vertices, $ax$, coincide: 

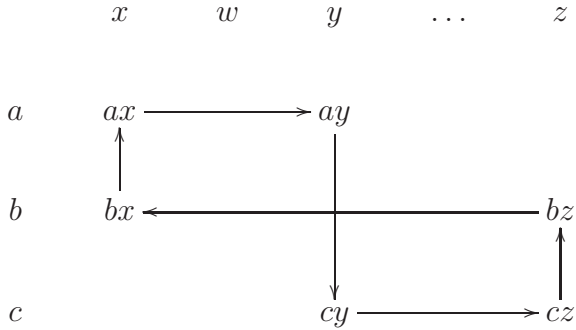
\begin{center}
\xymatrix{
       & x  &w & y &\ldots  & z\\
      a& ax \ar[rr]&  &ay \ar[dd] \\
      b& bx \ar[u]&  &    &&bz  \ar[llll]\\
      c&    &  &cy \ar[rr]  &&cz \ar[u]\\      
    }
\end{center}
 
 With this terminology in hand we can state the following result that characterizes Rees matrix semigroups in which all entries belong to $\{0,1\}$. 

\begin{theorem}\cite{lallement}\label{lall}
Let $S=M^0[G;I,\Lambda;P]$ be a Rees matrix semigroup over regular $P$. Then the following are equivalent:
\begin{enumerate}
\item all the entries in $P$ are either $0$ or $1$;
\item if all the products at the vertices of a closed polygonal line of the Cayley table of $S$ are all but one equal to a non-zero element $m$ and the remaining product is not zero, then it is also equal to $m$;
\item there exists a subsemigroup $T$ of $S$ satisfying the following property: for every $i\in I$ and $\lambda \in \Lambda$ there exists one and only one $g\in G$ such that $(i,g,\lambda)\in T$. 
\end{enumerate}
\end{theorem}
The subsemigroup $T$ mentioned in the last part of the theorem contains all the idempotents of $S$ and the condition implies, in particular, that the product of two idempotents is an idempotent. 

We have all the auxiliary results needed to start proving the following result,
stating the equivalence of (a) and (b) in our main theorem. 

 \begin{theorem}\label{cornerstone}
 Let $G\le S_n$ be a primitive group and let $t\in T_n$ be a rank $2$ transformation. Then $\langle G,t\rangle\setminus G$ is generated by its own idempotents if and only if its rank $2$ maps induce a connected Rees matrix semigroup.
 \end{theorem}

 Fix $t$, a rank $2$ map such that $A_1t=a_1$ and $A_2t=a_2$. Let $G\le S_n$ be a primitive group. The Rees matrix semigroup induced by $S$ is $$S'=M^0[H;\{\{A_1,A_2\}g\mid g\in G\}, \{\{a_1,a_2\}g\mid g\in G\};P],$$ with $H$ and $P$ still undefined. Regarding $H$, in general, it is the set of all maps in $S$ that have a given kernel and a given image (with the image being a transversal for the kernel). As there are only two rank $2$ maps  
with given image and kernel, it follows that $H$ either is the trivial group or $S_2$.  Now, the primitivity of $G$ implies that  there exist $g,h\in G$ such that $a_1g\in A_1$ and $a_2g\in A_2$, and $a_1h\in A_2$ and $a_2h\in A_1$. Therefore, $H$ has at least two elements and hence (by the discussion above) $H=S_2$, the symmetric group on two points. Regarding $P$ we ignore how it looks like, but  we can assume that it is in Graham's normal form as by Theorem \ref{gform}  every matrix $P$ of a  Rees matrix semigroup can be normalized. Note also that the connectedness of the semigroup $S'$  means that the graph $\Gamma(\mathbb{H}_{S'})$ is connected, and this is equivalent  (given the particular nature of $\Lambda$ and $I$ in $S'$) to saying that it is connected  the bipartite graph whose vertices is the union of $\{\ker(t)g\mid g\in G\}$ and $\{Xtg\mid g\in G\}$, and two vertices (a set $I$ and a partition $Q$) are connected if  $I$ is a transversal for $Q$.

If $S$ is idempotent generated, then  $S'$ is also idempotent generated and hence, by Theorem  \ref{gray}, $S'$ is a connected Rees matrix semigroup. 
 The direct implication of Theorem \ref{cornerstone} follows.

Regarding the converse, let $t\in T_n$ be a rank $2$ transformation and let $G\le S_n$ be a  primitive group such that $H(G,2,\ker(t),Xt)$ is connected.  
It is known that $G$ synchronizes every rank $2$ map and hence the semigroup $\langle G,t\rangle\setminus G$ will have some maps of rank $2$ and all the constants. It is obvious that the constants are idempotent. Thus $S:=\langle G,t\rangle\setminus G$ will be idempotent generated if and only if every rank $2$ map in $S$ can be written as a product of idempotents of $S$; by Theorem \ref{gray}, it is enough to prove that in $S'$ the entries of $P$ generate $S_2$.   The matrix $P$ fails  to generate $S_2$ only if all its entries are $0$ and $1$. To prove that this does not happen, by Theorem \ref{lall}, we only need to prove that there exists one closed $l$-polygonal line in $S$ in which all vertices are non-zero,  and such that $(l-1)$-vertices have value $m$, while the remaining vertice has a different value.  This is what we prove now using the following \emph{gadget}.

Let $x\in X$, and recall that $Xt=\{a_1,a_2\}$ and $A_1t=\{a_1\}$, $A_2t=\{a_2\}$. 
\[
V(x):=\{y\in X \mid (\exists g\in G)\{a_1,a_2\}g=\{x,y\}\}. 
\]
Now we define the following relation: for all $x,y\in X$,
 \[
 x\sim_1 y \Leftrightarrow (\exists g \in G) \ x,y \in A_1g \ \&\ V(x)\cap A_2g\neq \emptyset \neq V(y)\cap A_2 g.
 \]
The notation $Y\perp \pi$ means that the pair $Y=(y_1,y_2)$ is a transversal for the ordered partition $\pi=(P_1,P_2)$, with $y_i\in P_i$. With this notation,  $x$ and $y$ are $\sim_1$-related if there exist $x_1,y_1\in X$ such that $\{x,x_1\},\{y,y_1\}\in \{a_1,a_2\}G$, and  $(x,x_1)\perp (A_1,A_2)g \perp (y,y_1)$, for some $g\in G$. The elements $x_1$ and $y_1$ that are linked through $\sim$ to $x$ and $y$ (respectively) will be denoted by $\overline{x}$ and $\overline{y}$. 

Similarly we define  
\[
 x\sim_2 y \Leftrightarrow (\exists g \in G) \ x,y \in A_2g \ \&\ V(x)\cap A_1g\neq \emptyset \neq V(y)\cap A_1 g.
 \]

As $G$ is $2$-Hc, either $\sim_1$ or $\sim_2$ is non-trivial, say it is $\sim_1$. From now on it will be just denoted by  $\sim$, and $x_0, y_0$, are two different $\sim$-related elements. Observe that a primitive $G$ has the strong $2$-ut property, for all partition of type $(n-1,1)$ and hence, by the general results above, $\langle G,t\rangle\setminus G$ is idempotent generated so that we only have to care about partitions of type $(n,m)$, with both $m,n>1$. 

It is clear that $x\sim y \Rightarrow xg\sim yg$, for all $g\in G$, and hence $\sim$ is a $G$-relation.  The primitivity of $G$ guarantees that $\{\{x_0,y_0\}g\mid g\in G\}$ is a connected graph.  

Since $\sim$ is non-trivial, it follows that in  $\langle G,t\rangle$ there exists a map 
\[
a=\left(\begin{array}{cc}
A_1&A_2\\
\alpha_1&\alpha_2
\end{array}\right)
\]such that $\alpha_1\in A_1$ and $\alpha_2\in A_2$, with $(\alpha_1,\alpha_2)g=(a_1,a_2)$, for some $g\in G$, and there exists also $\{b_1,b_2\}\in \{a_1,a_2\}G$ such that $(b_1,b_2)\perp (A_1,A_2)$; thus $\alpha_1\sim b_1$. By the primitivity of $G$, it follows that there exist elements $c_1,\ldots,c_k\in X$ such that 
\[
b_1\sim c_1\sim c_2\sim \ldots \sim c_k\sim b_2.
\]
Recall that given a partition $\pi=(P_1,P_2)$ and $x\in P_i$, the notation $[x]_P$ means the part of $P$ containing $x$, that is, $P_i$. 

We observe that if there exists in $\langle G,t\rangle\setminus G$ a map $b$ such that $P_1b=\{b_1\}$ and $P_2b=\{b_2\}$, then there exists a map $b'\in \langle G,t\rangle\setminus G$ such that $P_1b'=\{b_2\}$ and $P_2b'=\{b_1\}$; this is a consequence of the primitivity of $G$. The second observation is that if $x\sim y$, then there are permutations $g_1,g_2,h\in G$ such that $\{x,\overline{x}\}=\{a_1,a_2\}g_1$, $\{y,\overline{y}\}=\{a_1,a_2\}g_2$, and $(x,\overline{x})\perp ( A_1,A_2)h\perp (y,\overline{y})$. This implies that the maps 
\[\begin{array}{ccc}
\left(\begin{array}{cc}
[x]_{Ag}&[\overline{x}]_{Ag}\\
x&\overline{x}
\end{array}\right)
&\mbox{ and }&
\left(\begin{array}{cc}
[y]_{Ag}&[\overline{y}]_{Ag}\\
y&\overline{y}
\end{array}\right)
\end{array}
\] both belong to  $\langle G,t\rangle\setminus G$ (where $A=(A_1,A_2)$). We use these two observations and the sequence of $\sim$-related elements introduced above to define a sequence of rank $2$ maps in $\langle G,t\rangle$:
\[\begin{array}{cc}
t_1=\left(\begin{array}{cc}
[b_1,c_1]_{T_1}&[b_2,\overline{c_1}]_{T_2}\\
x_1&x_2
\end{array}\right)
&
t_2=\left(\begin{array}{cc}
[c_1,c_2]_{T_2}&[\overline{c_1},\overline{c_2}]_{T_2}\\
x_1&x_2
\end{array}\right)\\
\\
\ldots\\
\\
t_k=\left(\begin{array}{cc}
[c_{k-1},c_k]_{T_k}&[\overline{c_{k-1}},\overline{c_k}]_{T_k}\\
x_1&x_2
\end{array}\right)
&
t_{k+1}=\left(\begin{array}{cc}
[c_k,b_2]_{T_{k+1}}&[\overline{c_k},b_1]_{T_{k+1}}\\
x_1&x_2
\end{array}\right),
\end{array}
\] and yet another sequence of maps: 
\[\begin{array}{cc}
b=\left(\begin{array}{cc}
A_1&A_2\\
b_1&b_2
\end{array}\right)
&
b_1=\left(\begin{array}{cc}
A_1&A_2\\
c_1&\overline{c_1}
\end{array}\right)\\
\\
\ldots\\
\\
b_{k-1}=\left(\begin{array}{cc}
A_1&A_2\\
c_{k-1}&\overline{c_{k-1}}
\end{array}\right)
&
b_{k}=\left(\begin{array}{cc}
A_1&A_2\\
c_{k}&\overline{c_{k}}
\end{array}\right).
\end{array}
\]

It is clear that the sequence
\[
bt_1,b_1t_1,b_1t_2,b_2t_2,\ldots,b_kt_{k+1},b_1t_{k+1}
\]is a closed polygonal line whose vertices all evaluate to 
\[
\left(\begin{array}{cc}
A_1&A_2\\
x_1&x_2
\end{array}\right)
\] except the last one that yields
\[
\left(\begin{array}{cc}
A_1&A_2\\
x_2&x_1
\end{array}\right).
\]
 
 Therefore, the semigroup $S'$ admits a closed polygonal line failing the equivalent conditions of Theorem \ref{lall} so that the matrix $P$ in $S'$ has two non-zero entries, and hence the entries in $P$  generate $S_2$. Theorem \ref{cornerstone} is proved.

\section{From $2$-Hc to road closures}\label{roadclos}

We say that a
permutation group $G$ has the $2$-Hc property if every Houghton graph
$H(G,2,S,P)$, where $S$ is a $2$-set and $P$ a $2$-partition of the domain,
is connected. (Recall that this graph has vertex set $SG\cup PG$, and an
edge from $S'$ to $P'$ if $S'$ is a transversal for $P'$.)
Since these graphs can be exponentially large (the number of 
$2$-partitions is $2^{n-1}$), we translate the property into one which can be checked
by looking at the orbital graphs for $G$.

\begin{theorem}
Let $G$ be a finite transitive permutation group on $\Omega$. The following
two conditions are equivalent:
\begin{enumerate}\itemsep0pt
\item $G$ has the $2$-Hc property;
\item for every orbit $O$ of $G$ on $2$-sets of $\Omega$, and every maximal
block of imprimitivity $B$ for $G$ acting on $O$, the graph with vertex set
$\Omega$ and edge set $O\setminus B$ is connected.
\end{enumerate}
\end{theorem}

\paragraph{Remark} We call condition (b) the \emph{road closure condition}:
orbital graphs for primitive groups are connected, and the condition asserts
that the graph cannot be disconnected by deleting a block of imprimitivity
for the action of~$G$. In other words, thinking of the orbital graph as a
connected road network, it cannot be disconnected by closing the roads in a
block of imprimitivity. As a simple example of a primitive graph for which
this property fails, consider the square grid graph (Figure~\ref{f:grid}).
(Two points in the same row or column are joined: the automorphism group is
the non-basic primitive group $G=S_m\wr S_2$.) The action of $G$ on edges
has two blocks of imprimitivity, the horizontal edges and the vertical
edges: removing one block leaves a graph with $m$ components.

\begin{figure}[htbp]
\begin{center}
\setlength{\unitlength}{1mm}
\begin{picture}(40,40)
\multiput(0,0)(10,0){5}{\circle*{1}}
\multiput(0,10)(10,0){5}{\circle*{1}}
\multiput(0,20)(10,0){5}{\circle*{1}}
\multiput(0,30)(10,0){5}{\circle*{1}}
\multiput(0,40)(10,0){5}{\circle*{1}}
\multiput(0,0)(10,0){5}{\line(0,1){40}}
\multiput(0,0)(0,10){5}{\line(1,0){40}}
\end{picture}
\end{center}
\caption{\label{f:grid}A grid fails the road closure condition}
\end{figure}
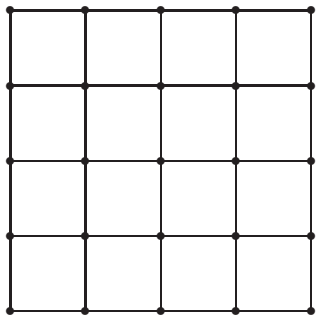

\begin{proof} 
Connectedness of $H(G,2,S,P)$ is equivalent to connectedness of the graph
with vertex set $SG$, having an edge from $S'$ to $S''$ whenever there is a
partition $P'\in PG$ for which both $S'$ and $S''$ are sections. We call this
the \emph{$2$-step Houghton graph}. This holds since every partition $P'$ is
joined to a subset $S'$.

Suppose that the $2$-Hc condition fails, and let $S$ and $P$ be a subset and
partition witnessing the failure. The edge set of a connected component of
the $2$-step Houghton graph is a block of imprimitivity $B$ for $G$ acting
on $SG$, since $G$ must permute the connected components among themselves.
Then, with $O=SG$, we see that $O\setminus B$ must have the property that no
edge is a section for $P$, and so this set is the edge set of a disconnected
graph (the parts of $P$ are unions of connected components).

Conversely, suppose that there is a $2$-set $S$ and a block $B$ for $G$ acting
on $O=SG$ such that the graph $(\Omega,O\setminus B)$ is disconnected. Let
$P$ be a $2$-partition, one of whose parts is a connected component for this
graph. Then every pair in $SG$ which is a section for $P$ must belong to $B$.
Hence all the edges of the $2$-step Houghton graph $H(G,2,S,P)$ are contained in
translates of $B$, and the graph is disconnected.

So the $2$-Hc property is equivalent to the road closure property.\qed
\end{proof}

A number of corollaries follow easily from this theorem. We begin with negative
results. 
The group $\pom^+(8,q)$ acts on a polar space which contains equal
numbers of points and of ``solids'' ($3$-dimensional projective spaces) in
each of two families; these are permuted transitively by the ``triality''
group of outer automorphisms, which induces $S_3$ on the three types of
object. The action of $\pom^+(8,q):S_3$ on triples of mutually incident
objects consisting of a point and a solid from each family is primitive
(these are examples $P_2$ in \cite[Table III]{kleid}, see also \cite{bhrd}).\label{trial}

\begin{theorem}\label{nonex}
\begin{enumerate}\itemsep0pt
\item A transitive imprimitive group fails the road closure property.
\item A primitive non-basic group fails the road closure property.
\item A primitive group which has an imprimitive normal subgroup of index~$2$
fails the road closure property.
\item The primitive action of $\pom^+(8,q):S_3$ described above
fails the road closure property.
\end{enumerate}
\end{theorem}

\begin{proof} (a) If $S$ is a $2$-subset of a block of imprimitivity,
then the graph with edge set $SG$ is disconnected.

\smallskip

(b) Suppose that $G$ is primitive but non-basic; then $\Omega$ can be identified
with the set $Q^d$ of all $d$-tuples over an alphabet $Q$ of size $q>2$,
and $G$ is contained in the wreath product $\mathrm{Sym}(Q)\wr S_d$, where
the group permuting the coordinates is transitive.

Consider a pair $S$ of points which agree in all but one coordinate. The
images of $S$ under $G$ contain, for each coordinate, a pair of vertices which
differ only in that coordinate. Now, for each fixed coordinate, the pairs
differing in that coordinate form a block of imprimitivity $B$; and
the graph with edge set $SG\setminus B$ is disconnected, since all vertices
in a connected component have the same entry in the chosen coordinate.

\smallskip

(c) Suppose that the primitive group $G$ has an imprimitive normal subgroup
$H$ of index~$2$. Let $B$ be a block for $H$ containing a point $x$, and
choose $y$ in $B$; now put $S=\{x,y\}$, and let $P$ be the partition 
$(B,X\setminus B)$. Now all the images of $S$ under $H$ are subsets of blocks
in the block system containing $B$, so $SH$ is the edge set of a disconnected
graph; and $SG=SH\cup SHg$ for $g\in G\setminus H$, so $SH$ is a block for
$G$ acting on $SG$, and $SG\setminus SH$ is disconnected. Thus the road
closure property fails.

\smallskip

(d) Let $G=\pom^+(8,q):S_3$, and $H=\pom^+(8,q)$, a normal subgroup of 
index~$6$ in $G$ with quotient group $S_3$.

Let $t=(p,\sigma,\sigma')$ be a triple belonging to the set on which $G$ acts,
and let $t'$ be another triple having two elements in common with $t$.
Let $S=\{t,t'\}$. Then $SG$ falls into three orbits under $H$, which are
blocks of imprimitivity for $G$ in its action on $SG$: each orbit is determined
by one of the three positions in the triple where its elements disagree. 

Consider the graph whose edge set is the union of two of these three blocks,
say those corresponding to disagreement in the second and third positions.
Then any edge joins triples which agree in the first position, so the
entire connected component consists of triples which agree in the first
position. So the graph with this union of blocks as edges is disconnected.\qed
\end{proof}

\paragraph{Remarks} (a) There are several examples of primitive groups 
satisfying the conditions of (c) of Proposition~\ref{nonex}. Such a group is
contained in the automorphism group of an incidence structure of points and
blocks, acting on the set of flags (incident point-block pairs). To see this,
we may choose $B$ to be a minimal block of imprimitivity for $H$; then, if
$x\in B$ and $g\in Gx\setminus H_x$, $Bg$ is another block of imprimitivity
for $H$, and so $B\cap Bg=\{x\}$. Now we construct the incidence structure as
follows: its ``points'' are the $H$-translates of $B$, and the ``blocks'' the
$H$-translates of $Bg$, a ``point'' and ``block'' being incident if their
intersection is non-empty. We see that there is a bijection between the domain
of $G$ and the set of flags of the incidence structure.

Examples include
\begin{enumerate}\itemsep=0pt
\item points and hyperplanes of a finite projective space, where incidence is
inclusion (so that $G$ acts on the set of point-hyperplane flags);
\item more generally, the $i$-spaces and $n-i-1$-spaces in $n$-dimensional
projective space over a finite field, where incidence is inclusion;
\item points and hyperplanes of a finite projective space, where incidence
is non-inclusion;
\item points and lines of a self-dual generalized quadrangle (the symplectic
quadrangle over a field of characteristic~$2$);
\item points and lines of a self-dual generalised hexagon (associated with
the group $G_2(q)$, where $q$ is a power of~$3$);
\item points and blocks of a suitable symmetric design such as the
$(11,5,2)$ or $(11,6,3)$ designs (these give examples of degrees
$55$ and $66$, with $G=\pgl(2,11)$), or the $(176,126,90)$ design
associated with the Higman--Sims group).
\end{enumerate}

\smallskip

Another class of examples, extending the example above of $\pgl(2,11)$
with degree~$55$, is given by the following construction.

Let $p$ be a prime congruent to $\pm1$ (mod~$5$) and to $\pm3$ (mod~$8$).
From the list of subgroups of $\pgl(2,p)$ (e.g.\ in Dickson~\cite{dickson}
or Huppert~\cite{huppert}),
we see that $\pgl(2,p)$ contains one conjugacy class of subgroups isomorphic
to $A_5$, splitting into two classes in $\psl(2,p)$. Now an $A_4$ inside one
of these $A_5$s is normalised by $S_4$ in $\pgl(2,p$); elements of
$S_4\setminus A_4$ thus conjugate $A_5$ into another $A_5$ (in the other
class in $\psl(2,p)$) intersecting it in $A_4$.

So $\pgl(2,p)$ acts primitively on the cosets of $S_4$, but the subgroup of
index $2$ (namely $\psl(2,p)$) is imprimitive, since the stabiliser $A_4$ is
contained in two $A_5$s (one in each class) -- it has two systems of blocks
of size $5$.

This gives an action of $\pgl(2,p)$ of degree $p(p^2-1)/24$.

\medskip

We come now to our main conjecture, which asserts that the converse of this
theorem is true.

\begin{conj}
A primitive basic permutation group which does not satisfy condition (c) or (d)
of Theorem~\ref{nonex1} has the road closure property.
\end{conj}

Computation shows that the conjecture is true for groups with degree at most
$130$, as we will discuss later. We tried to prove this conjecture, and proposed it to some world class experts in permutation groups, but after many attempts by several people, all of us formed the conviction that this is an extremely difficult problem.

\section{Some positive results}\label{pos}

Next we give some examples to show that many ``typical'' primitive
groups do have the road closure property.

\begin{theorem}
\begin{enumerate}\itemsep0pt
\item A $2$-homogeneous group has the road closure property.
\item A transitive permutation group of prime degree has the road closure
property.
\item A primitive permutation group of degree the square of a prime has the
road closure property if and only if it is basic.
\item The symmetric or alternating group of degree $m$, acting on the set of
$k$-element subsets of $\{1,\ldots,m\}$ (with $m>2k$), has the road closure
property.
\end{enumerate}
\end{theorem}

\begin{proof}
(a) Suppose that $G$ is $2$-homogeneous. Then, for any $2$-set $S$, the graph
with edge set $SG$ is the complete graph.

Let $B$ be a block of imprimitivity for $G$ acting on $2$-sets (possibly a
singleton). If
$(\Omega,SG\setminus B)$ is disconnected, then $(\Omega,B)$ would be
connected, and so $(\Omega,B')$ would be connected for any translate $B'$
of $B$. But this is impossible, since $SG\setminus B$ is the union of all the
other translates of $B$.

\smallskip

(b) According to Burnside's Theorem, a transitive group $G$ of prime degree $p$
is either $2$-transitive (in which case part (a) applies), or is a subgroup of
$\agl(1,p)$. If a non-trivial block of imprimitivity
for $G$ on an orbit of $2$-sets has size divisible by $p$, then all its
translates contain connected circulant graphs, and so the complement is of
the original set is the edge set of a connected graph.
A block of size coprime to $p$ meets a $p$-cycle in one point,
and so its complement contains a path of length $p$ and is connected.

\smallskip

(c) Of course we may assume that $p$ is odd.

We begin with Wielandt's theorem \cite[Theorem 16.2]{wie},
which asserts that a primitive permutation group $G$ of degree $p^2$ satisfies
one of the following:
\begin{itemize}\itemsep0pt
\item[(i)] $G$ is affine;
\item[(ii)] $G\le S_p\wr S_2$;
\item[(iii)] $G$ is $2$-transitive.
\end{itemize}

Clearly type (ii) are non-basic (and hence fail the road closure property),
while type (iii) are basic and have the property. So we may assume that
$G$ is affine. Thus $G=p^2:H$, where
the linear group $H$ acts irreducibly on the vector space $V$ representing
the $p^2$. Irreducibility means simply that $H$ fixes no $1$-dimensional
subspace of $V$. Also, $G$ is basic if and only if $H$ is a primitive
linear group, which means that $H$ has no orbit of size $2$ on the set of
$1$-dimensional subspaces of $V$.

Now if $G$ is non-basic, then it does not have the road closure property.
So we may assume that
$G$ is basic, which (as above) means that the subgroup of $\pgl(2,p)$
induced by $H$ has no orbit of length $1$ or $2$ on the projective line.
All subgroups of $\pgl(2,p)$ are known, and we could simply examine individual
groups. Instead, the following argument aims at some generality.

We have to show that, for any orbital graph for $G$ (with edge set $O$), and
any block of imprimitivity $B$ for $G$ in its action on $O$, the graph with
edge set $O\setminus B$ is connected. Any edge $xy$ has a ``direction'', a
point on the line at infinity corresponding to the subspace of $V$ spanned by
$y-x$.

The graph with edge set $O$ is a Cayley graph for the translation group of $V$;
under this group, $V$ splits into orbits of size $p^2$, each of which is a
union of $p$ cycles of length $p$. We note that two subspaces of $V$
corresponding to different directions give a grid structure to $V$; if we
choose elements in these two subspaces, the resulting Cayley graph is the
Cartesian product of two cycles, and so is connected.

Suppose that $B$ is a union of $V$-orbits. There are two possibilities.
It may be that any two orbits whose edges have the same direction lie in the
same block. Then there are at least two directions outside $B$ realised by
$O$, and so $O\setminus B$ is connected. On the other hand, it may be that $B$
contains some but not all of the edges in each direction realised by $O$.
Then it also avoids at least one edge in each such direction, and again
$O\setminus B$ is connected.

So we may assume that $B$ is not a union of $V$-orbits. Now the intersection
of $B$ with a $V$-orbit is a block of imprimitivity for $V$. There are three
cases: the intersection has cardinality $1$; it has cardinality $p$ but
contains one edge from every cycle of the element of $V$ corresponding to the
direction of an edge; or it consists of a cycle of an element of $V$. In the
first two cases, we can choose $p-1$ edges in that direction forming a path.
Doing this in two different directions, we find the Cartesian product of
two paths, and is connected.

In the final case, the edges of $B$ in some fixed direction form a cycle of
an element of $V$, and so lie in a line of the affine plane. We can assume
that this is true for every direction. So $O\setminus B$ contains edges in
$p-1$ of the $p$ cycles in each possible direction. The only way to avoid
connectedness of $O\setminus B$ is that the omitted lines all pass through
the same point $x$, which thus has the property that $B$ is the set of all
edges containing $x$. But this is impossible. For if $xy$ is such an edge,
then the set of edges containing $y$ would also be a block $B'$; but
$B\cap B'=\{xy\}$, a contradiction.

\smallskip

(d) Let $G$ be $S_m$ acting on $k$-sets, with $m>2k$. Now a pair
$S$ of $k$-sets intersecting in $l$ points is stabilised by the direct product
of $S_l$, $S_{k-l}\wr S_2$, and $S_{m-2k+l}$. The overgroups in the symmetric
group are easily computed. We see that the possible blocks of imprimitivity
containing $S$ for $G$ acting on $SG$ consist of all pairs with the same
intersection, all pairs with the same symmetric difference, or all pairs with
the same union. In each of these cases, it is easy to see that the complement
of a block in the orbital graph is connected; the relevant set of pairs can
easily be bypassed.\qed
\end{proof}

\section{Computational results}\label{comput}

We have tested all primitive groups of degree up to $130$, and a number of 
groups of larger degree, and found no counterexample to our conjecture.

The algorithm checks the road closure condition in the simplest possible way.
Given a primitive group $G$, we do the following:
\begin{enumerate}\itemsep0pt
\item Check if $G$ is basic (the road closure fails if not).
\item Compute the orbits of $G$ acting on the set of $2$-element subsets.
\item For each orbit, compute the maximal blocks of imprimitivity for the
action of $G$ on this orbit; remove a block and check the remaining graph for
connectedness.
\end{enumerate}
We make a few comments on each step.

For the first step, we may make use of the \textsf{GAP} function
\texttt{ONanScottType} to exclude the non-basic groups. Unfortunately, this
does not work for affine primitive groups, since the function does not analyse
them further. So we had to write our own test for the non-basic property of
an affine group: build the possible Hamming graphs, ahd check for each union
of $G$-orbits on $2$-sets whether the corresponding graph is isomorphic to
a Hamming graph. One thing on our wish list for \textsf{GAP} is a test for
the basic property which works for affine groups!

The second step is straightforward.

For the third step, there is a \textsf{GAP} command to find all the blocks
of imprimitivity for a transitive permutation group containing a given point
of the domain. This command can take some time. It is known that the minimal
blocks of imprimitivity can be found in polynomial time~\cite{atkinson}; the
procedure for finding all blocks involves finding all the minimal blocks, and
for each such block, find all minimal blocks for the group acting on the
corresponding block system, and so on until we reach the system with a single
block.

This raises an interesting theoretical question.

\begin{qn}
Is there a polynomial upper bound in terms of $n$ for the number of maximal
blocks of imprimitivity (containing a given point) of a transitive permutation
group of degree~$n$?
\end{qn}

A special case of this question is the famous conjecture of Wall~\cite{wall},
according to which the number of maximal subgroups of a finite group is at
most the order of the group. (If a group $G$ has its regular action, then the
blocks of imprimitivity containing the identity are just the subgroups of $G$.)
Wall's conjecture is known to be false, but Liebeck \emph{et~al.}
\cite{lps_wall} found an upper bound of order $|G|^{3/2}$.

Some improvements to the program involve excluding groups dealt with by other
means such as those discussed above. We wrote a program along the lines just
described, and used it to find the basic groups failing the road closure
property up to degree $130$ and for several larger degrees. No counterexamples
to our conjecture were found. In the table, we give the degree, the number in
the list in \textsf{GAP} 4.7.4, and the name of the group.

\begin{table}[htbp]
\[
\begin{array}{|r|r|l|}
\hline
\hbox{Degree} & \hbox{Number} & \hbox{Group} \\
\hline
21 & 1 & \psl(3,2):2 \\
28 & 1 & \psl(3,2):2 \\
45 & 1,2,3 & S_6:2\hbox{ and subgroups} \\
52 & 1 & \psl(3,3):2 \\
55 & 3 & \psl(2,11):2 \\
66 & 1 & \psl(2,11):2 \\
105 & 1,\ldots,6 & \mathrm{Aut}(\psl(3,4))\hbox{and subgroups} \\
105 & 7 & S_8=\psl(4,2):2 \\
117 & 1 & \psl(3,3):2 \\
120 & 11 & S_8=\psl(4,2):2 \\
120 & 1 & S_7 \\
\hline
\end{array}
\]
\caption{\label{non-rc}Basic primitive groups without road closure property}
\end{table}

\section{Problems}

We start this section asking two of the most important questions prompted by this paper.
\begin{problem}
Is there any relation between being $k$-homogeneous and possessing the strong $k$-ut property?
\end{problem}

\begin{problem}
Is the road closure conjecture true?
\end{problem}

The classification of groups with $k$-id is almost finished, but there is still some cases to decide; probably, it is necessary to devlope more robust GAP code (more on that below).
\begin{problem}
Finish the classification of permutation groups that have the $3$-id property. The same for the $4$-id property. In particular, does
$\pgaml(2,32)$ have the $5$-id property? Does $M_{11}$ (degree~$12$) have the
$4$-id property?
\end{problem}

\begin{problem}
Is there a combinatorial condition on a permutation group which
is necessary and sufficient to the $k$-id property for $k>2$, analogous to the
road closure property for the $2$-id property?
\end{problem}

Recall that a group is said to be synchronizing if together with any singular map generates a constant. 

\begin{problem}
Let $G\le S_n$ be a primitive group and $t$ a non-invertible map. Is it true that the subsemigroup of $\langle G,t\rangle$ formed by its maps of minimum rank is generated by idempotents? This is trivially true for synchronizing groups; the question is what happens for the other primitive groups.   
\end{problem} 

As said in the introduction, the origins of this research are in two results, one proved by Howie to transformations on a set, and another similar proved by Erdos for transformations on a vector space. This similitude between the two semigroups is well known and studied, and in the context of this paper the following problem is very natural.

\begin{problem}
Classify the linear groups $G$ that together with any non-invertible linear transformation $t$ yield an idempotent generated semigroup:
\[	
\langle G,t\rangle \setminus G=\langle E\rangle . 
\]
\end{problem}

This paper closes the project started in \cite{ac_tams}. Now the next step is the following problem. 

\begin{problem}
Let $\Omega$ be a finite set. 
\begin{itemize}	
\item Classify the pairs $(G,I)$, where $G\le S_n$ and $I\subseteq \Omega$, such that the semigroup generated by  $G$ and any map with mage $I$ is regular. 
\item Classify the pairs $(G,I)$, where $G\le S_n$ and $I\subseteq \Omega$, such that the semigroup of singular maps generated by  $G$ and any map with mage $I$ is idempotent generated. 
\end{itemize} 
\end{problem}

To handle the undecided questions, and for general use, it would be convenient to have in GAP a number of new functions based on effective algorithms.  
\begin{problem}
\begin{enumerate}
\item Provide a command that finds if a $0$-Rees Matrix Semigroup is connected;
\item Provide a command that returns a given $0$-Rees Matrix Semigroup in Graham's normal form; 
\item  Given the two commands above, then it should be very easy (using Theorem \ref{gray}) to check if a $0$-Rees Matrix Semigroup is idempotent generated or not. 
\item Produce more efficient code to check if a permutation group has the $k$-id. Observe that with the code available we could not check if $M_{11}$ (degree $12$) has the $3$-id. 
\end{enumerate}
\end{problem}	

Regarding GAP functions to handle groups, we need very effective algorithms for the following:
\begin{problem}
\begin{enumerate}
\item check if a group has the [strong] $k$-ut property; 
\item find the sets $S\subseteq \Omega$ such that in the orbit of $S$ there is a transversal for every $|S|$-partition. 
\end{enumerate}
\end{problem}

\end{document}